\documentclass{amsart}
\usepackage{epsfig}
\usepackage{amsfonts,amssymb,amscd,amsmath,euscript,enumerate,verbatim,calc}
\theoremstyle{plain}

\newtheorem{theorem}{Theorem}[section]
\newtheorem{lemma}[theorem]{Lemma}
\newtheorem{proposition}[theorem]{Proposition}

\newtheorem{corollary}[theorem]{Corollary}
\theoremstyle{definition}
\newtheorem{definition}[theorem]{Definition}

\newtheorem{remark}[theorem]{Remark}

\def\Fq{{\mathbb F}_q}
\def\Fqm{{\mathbb F}_{q^m}}

\newcommand{\GL}{\operatorname{GL}}
\newcommand{\G}{\operatorname{G}}
\newcommand{\TSRI}{\operatorname{TSRI}}

\newcommand{\TSR}{\operatorname{TSR}}

\begin{document}

\title[An Asymptotic Formula for the number of Irreducible TSRs]
{An Asymptotic Formula for the number of Irreducible Transformation Shift Registers}
\author{Stephen D. Cohen}
\address{School of Mathematics and Statistics, University of Glasgow \newline \indent Glasgow G12 8QW, Scotland}
\email{Stephen.Cohen@glasgow.ac.uk}
\author{Sartaj Ul Hasan}
\address{Scientific Analysis Group, Defence Research and Development
Organisation \newline \indent
Metcalfe House, Delhi 110054, India}
\email{sartajulhasan@gmail.com}
\author{Daniel Panario}
\address{School of Mathematics and Statistics,
Carleton University, Ottawa, K1S 5B6, Canada}
\email{daniel@math.carleton.ca}
\author{Qiang Wang}
\address{School of Mathematics and Statistics,
Carleton University, Ottawa, K1S 5B6, Canada}
\email{wang@math.carleton.ca}



\begin{abstract}
We consider the problem of enumerating the number of irreducible
transformation shift registers. We give an asymptotic formula
for the number of irreducible transformation shift registers in some special cases. Moreover, we derive a short proof for the exact number
of irreducible transformation shift registers of order two using
a recent generalization of a theorem of Carlitz.
\end{abstract}
\date{\today}
\maketitle

\section{Introduction}
Linear feedback shift registers (LFSRs) are devices that are
used to generate sequences over a finite field. This sort of
sequence has received numerous applications in various
disciplines including in the design of stream ciphers; see,
for example, \cite{GG,LN}. For all practical purposes, these
sequences are generally considered over a binary field. The
sequences with maximal period have been proved to have good
cryptographic properties. LFSRs corresponding to sequences
with maximum period are known as primitive LFSRs.

The number of primitive LFSRs of order $n$ over a finite field
$\Fq$ is given by
\begin{equation} \label{NoLFSR}
\frac {\phi(q^{n}-1)}{n},
\end{equation}
where $\phi$ is Euler's totient function. A similar formula
for the number of irreducible LFSRs (that is, when the
characteristic polynomial of the LFSR is irreducible) of
order $n$ over a finite field $\Fq$ is given by
\begin{equation} \label{NoIrrLFSR}
 {\frac{1}{n} \displaystyle \sum_{d\mid n}
              \mu\left(d\right)q^{\frac{n}{d}}},
\end{equation}
where $\mu$ is the M\"{o}bius function.

Niederreiter \cite{N2} introduces the notion of \emph{multiple
recursive matrix method}, which may be considered as a
generalization of the classical LFSRs. 
Zeng et.~al \cite{Zeng} consider the notion of $\sigma$-LFSR
which is a word-oriented stream cipher. It turns out that the latter is essentially same as Niederreiter's multiple recursive matrix method.
A conjectural formula for the number of
primitive $\sigma$-LFSRs of order $n$ was given in the binary
case in \cite{Zeng}. An extension of this conjectural formula
over the finite field $\Fqm$ given in \cite{GSM} states that
this number is
\begin{equation} \label{NoSigmaLFSR}
   \frac{\phi(q^{mn}-1)}{mn} q^{m(m-1)(n-1)}
        \displaystyle \prod_{i=1}^{m-1}(q^m-q^i).
\end{equation}
We refer to \cite{GSM} and \cite{GR} for recent progress on
this conjecture and to \cite{CT} for a proof of this conjecture.

It is also known from \cite{GR} and \cite{CT}, see also
\cite{Ram}, that the number of irreducible $\sigma$-LFSRs is 
\begin{equation}\label{NoIrrSigmaLFSR}
  {\displaystyle \frac{1}{mn}{q^{m(m-1)(n-1)}
   \displaystyle \prod_{i=1}^{m-1}(q^m-q^i)} \sum_{d\mid mn}
        \mu\left(d\right)q^{\frac{mn}{d}}}.
\end{equation}

We focus  on \emph{transformation shift registers}
(TSRs) in this paper. This notion was introduced by Tsaban and Vishne \cite{TV}
and it can be also considered as a generalization of classical
LFSRs. The notion of TSR was introduced to address a problem of
Preneel \cite{BP} on designing fast and secure LFSRs with the help
of the word operations of modern processors and the techniques of
parallelism. It may be noted that the family of TSRs is a subclass
of the family of $\sigma$-LFSRs. Dewar and Panario [\ref{DP1},
\ref{DP3}] further studied the theory of TSRs.

We do not know yet any explicit formula like (\ref{NoLFSR}) and
(\ref{NoSigmaLFSR}) for the number of primitive TSRs. The problem
of enumerating primitive TSRs was first considered in \cite{HPW}. It was
proved that in order to count primitive TSRs, it is sufficient to
enumerate certain block companion matrices in a corresponding general linear
group. However, except few initial cases, this problem seems
rather difficult and still remains open.

Based on some empirical evidence, Tsaban and Vishne \cite{TV}
pointed out that irreducible TSRs contain a high proportion of
primitive TSRs. Thus in order to find a primitive TSR in practice
one might try an exhaustive search only among the irreducible
ones instead of over all TSRs; there is a high chance that
one might end up getting a primitive TSR in this way. This
reduces the search complexity of primitive TSRs. Motivated by
this fact and in an attempt to obtain a nice formula like
(\ref{NoIrrLFSR}) and (\ref{NoIrrSigmaLFSR}), we consider here
the problem of enumerating irreducible TSRs. In fact, this
problem was first considered in \cite{Ram} where the author
gives a formula for the number of irreducible TSRs of order two.
Moreover, in \cite{Ram}, as a consequence of this result, a
new proof of a theorem of Carlitz about the number of the self
reciprocal irreducible monic polynomials of a given degree over
a finite field is deduced.

Our paper is organized in the following manner.
In Section \ref{sectsr} we recall
some results concerning transformation shift registers needed
in this work.

As it has been mentioned earlier, Ram \cite{Ram} gives a formula
for the number of irreducible TSRs of order two. In Section
\ref{TSRorder2} we give a short proof of Ram's result using a
variant of a theorem of Carlitz recently proved \cite{Ahmadi}. Asymptotic analysis of the number of irreducible TSRs of order two is carried out in Section \ref{Asymptotic}.
Finally, in Section \ref{Asymptoticodd}, we prove an asymptotic
formula for the number of irreducible TSRs of any order when $q$ is odd.

\section {Transformation Shift Registers} \label{sectsr}

We denote by $\Fq$ the finite field with $q=p^r$ elements,
where $p$ is a prime number and $r$ is a positive integer,
and by $\Fq[X]$ the ring of polynomials in one variable $X$
with coefficients in $\Fq$. For every set $S$, we denote by
$|S|$, the cardinality of the set $S$. Also we denote by
$M_d(\Fq)$, the set of all $d\times d$ matrices with entries
in $\Fq$. We now recall from \cite{HPW} some definitions and
results concerning transformation shift registers.

Throughout this and subsequent sections, we fix positive
integers $m$ and $n$, and a vector space basis $\{\alpha_0,
\dots, \alpha_{m-1}\}$ of ${\mathbb F}_{q^m}$ over
$\mathbb F_q$. Given any $ s\in {\mathbb F}_{q^m}$, there are
unique $s_0, \dots, s_{m-1} \in {\mathbb F}_{q}$ such that
$s= s_0 \alpha_0 + \cdots + s_{m-1}\alpha_{m-1}$,
and we shall denote the corresponding co-ordinate vector
$(s_0, \dots, s_{m-1})$ of $s$ by $\mathbf{s}$. Evidently,
the association $s\longmapsto \mathbf{s}$ gives a vector
space isomorphism of $\mathbb F_{q^m}$ onto $\mathbb F_q^m$.
Elements of $\mathbb F_q^m$ may be thought of as row vectors
and so $\, \mathbf{s}C$ is a well-defined element of
$\mathbb F_q^m$ for any $\mathbf{s} \in \mathbb F_q^m$ and
$C\in M_m(\Fq)$.

\begin{definition} \label{def:tsr} 
Let $c_0, c_1, \dots, c_{n-1} \in \Fq$ and $ A \in M_m(\mathbb F_q)$.
Given any $n$-tuple $(\mathbf{s}_0, \dots, \mathbf{s}_{n-1})$ of
elements of  $\mathbb F_{q^m}$, let $(\mathbf{s}_i)_{i=0}^{\infty}$
denote the infinite sequence of elements of ${\mathbb F}_{q^m}$
determined by the following linear recurrence relation:
\begin{eqnarray}
{\mathbf{s}}_{i+n}={\mathbf{s}}_i(c_0A)+{\mathbf{s}}_{i+1}(c_1A)
          +\cdots +{\mathbf{s}}_{i+n-1}(c_{n-1}A)
          \quad i=0,1,\dots. \label{tsr}
\end{eqnarray}
The system (\ref {tsr}) is a \emph{transformation shift register}
(TSR) of order $n$ over $\mathbb F_{q^m}$, while the sequence
$(\mathbf{s}_i)_{i=0}^{\infty}$ is the \emph{sequence generated
by the TSR} (\ref{tsr}). The $n$-tuple $(\mathbf{s}_0,\mathbf{s}_1,
\ldots, \mathbf{s}_{n-1})$ is the \emph{initial state} of the TSR
(\ref{tsr}) and the polynomial $I_mX^n -(c_{n-1}A)X^{n-1}- \cdots
-(c_1A)X-(c_0A)$ with matrix coefficients is the \emph{tsr-polynomial}
of the TSR (\ref{tsr}), where $I_m$
denotes the $m\times m$ identity matrix over $\Fq$. The sequence $(\mathbf{s}_i)_{i=0}^{\infty}$
is \emph{ultimately periodic} if there are integers $r, n_0$ with
$r\ge 1$ and $n_0\geq0$ such that $\mathbf{s}_{j+r}=\mathbf{s}_j$
for all $j \geq n_0$. The least positive integer $r$ with this
property is the \emph{period} of $(\mathbf{s}_i)_{i=0}^{\infty}$
and the corresponding least nonnegative integer $n_0$ is the
\emph{preperiod} of $(\mathbf{s}_i)_{i=0}^{\infty}$. The sequence
$(\mathbf{s}_i)_{i=0}^{\infty}$ is \emph{periodic} if its preperiod
is $0$.
\end{definition}

The following proposition gives some basic facts about TSRs.
\begin{proposition} \cite{HPW} \label{tsrulti}
For the sequence $(\mathbf{s}_i)_{i=0}^{\infty}$ generated by
the TSR $(\ref{tsr})$ of order $n$ over $\mathbb F_{q^m}$, we have
\begin{enumerate}
\item[{\rm (i)}] $(\mathbf{s}_i)_{i=0}^{\infty}$ is ultimately
periodic, and its period is no more than $q^{mn}-1$;
\item[{\rm (ii)}] if $c_0 \neq 0$  and $A$ is nonsingular, then
$(\mathbf{s}_i)_{i=0}^{\infty}$ is periodic; conversely, if
$(\mathbf{s}_i)_{i=0}^{\infty}$ is periodic whenever the
initial state is of the form $(b, 0, \dots , 0)$, where
$b\in \mathbb F_{q^m}$ with $b\ne 0$, then $c_0A$ is nonsingular.
\end{enumerate}
\end{proposition}

A TSR of order $n$ over $\mathbb F_{q^m}$ is \emph{primitive} if
for any choice of nonzero initial state, the sequence generated
by that TSR is periodic of period  $q^{mn}-1$.

Corresponding to a \emph{tsr-polynomial} $I_mX^n -(c_{n-1}A)X^{n-1}-
\cdots-(c_1A)X-(c_0A) \in M_m(\Fq)[X]$, we can associate a $(m,n)$-block
companion matrix $T \in M_{mn}(\Fq)$ of the following form
\begin{equation} \label{typeP}
T =
\begin {pmatrix}
\mathbf{0} & \mathbf{0} & \mathbf{0} & . & . & \mathbf{0} & \mathbf{0} & c_0A\\
I_m & \mathbf{0} & \mathbf{0} & . & . & \mathbf{0} & \mathbf{0} & c_1A\\
. & . & . & . & . & . & . & .\\
. & . & . & . & . & . & . & .\\
\mathbf{0} & \mathbf{0} & \mathbf{0} & . & . & I_m & \mathbf{0} & c_{n-2}A\\
\mathbf{0} & \mathbf{0} & \mathbf{0} & . & . & \mathbf{0} & I_m & c_{n-1}A
\end {pmatrix},
\end{equation}
where $c_0, c_1, \dots , c_{n-1}\in \Fq$, $A\in M_m(\Fq)$ and 
$\mathbf{0}$
indicates the zero matrix in $M_m(\Fq)$. The set of all such
$(m,n)$-block companion matrices $T$ over $\Fq$ shall be denoted
by $\TSR(m,n,q)$. Using a Laplace expansion or a suitable sequence of
elementary column operations, we conclude that if $T \in \TSR(m,n,q)$
is given by \eqref{typeP}, then $\det T = \pm \det (c_0A)$.
Consequently,
\begin{equation}
\label{nonsingP}
T \in \GL_{mn}(\Fq) \Longleftrightarrow
    c_0 \neq 0 ~~\mbox{and}~~A\in \GL_m(\Fq).
\end{equation}
where $\GL_m(\Fq)$ is the general linear group of all $m \times m$
nonsingular matrices over $\Fq$.

It may be noted that the block companion matrix \eqref{typeP} is the
state transition matrix for the TSR \eqref{tsr}. Indeed, the $k$-th
state $\mathbf{S}_k:=\left(\mathbf{s}_{k}, \mathbf{s}_{k+1}, \dots,
\mathbf{s}_{k+n-1}\right) \in \Fqm^n$ of the TSR (\ref{tsr}) is
obtained from the initial state
$\mathbf{S}_0:=\left(\mathbf{s}_{0}, \mathbf{s}_{1}, \dots,
\mathbf{s}_{n-1}\right) \in \Fqm^n$ by $\mathbf{S}_k = \mathbf{S}_0 T^k$,
for any $k\ge 0$.

In view of Proposition~\ref{tsrulti} and (\ref{nonsingP}), we have
that $T\in \TSR(m,n,q)$ is periodic if and only if $T$ has the
following form
\begin{equation} \label{typeP'}
\begin {pmatrix}
\mathbf{0} & \mathbf{0} & \mathbf{0} & . & . & \mathbf{0} & \mathbf{0} & B\\
I_m & \mathbf{0} & \mathbf{0} & . & . & \mathbf{0} & \mathbf{0} & c_1B\\
. & . & . & . & . & . & . & .\\
. & . & . & . & . & . & . & .\\
\mathbf{0} & \mathbf{0} & \mathbf{0} & . & . & I_m & \mathbf{0} & c_{n-2}B\\
\mathbf{0} & \mathbf{0} & \mathbf{0} & . & . & \mathbf{0} & I_m & c_{n-1}B
\end {pmatrix},
\end{equation}
where $c_1, \dots, c_{n-1}\in \Fq$ and $B\in \GL_m(\Fq)$. In what follows,
we deal with periodic TSRs only, that is, a TSR of the form (\ref{typeP'}).

The following lemma
reduces the calculation of an $mn\times mn$ determinant to an
$m\times m$ determinant.

\begin{lemma}
\label{mntom} \cite{HPW}
Let $T \in \TSR(m,n,q)$ be given as in \eqref{typeP'} and also let
$F(X)\in M_m\left(\Fq[X]\right)$ be defined by
$F(X) := I_m X^n - (c_{n-1}B)X^{n-1} - \cdots - (c_1B) X - B$.
Then the characteristic polynomial of $T$ is equal to
$\det \left(F(X)\right)$.
\end{lemma}
The following proposition entails that the problem of counting the
number of primitive TSRs is equivalent to the enumeration of certain
block companion matrices.
\begin{proposition} \label{conj}\cite{HPW}
Let $o(T)$ denote the period of the sequence generated by
$T \in \TSR(m,n,q)$. The number of primitive TSRs of order $n$
over $\Fqm$ is equal to the cardinality of the set
$$\left\{T \in \TSR(m,n,q)~:~ T~~\mbox{is of the form}
~~\eqref{typeP'}~~\mbox{and}~~o(T) = q^{mn}-1 \right\}.$$
\end{proposition}

The case $n=1$ follows immediately from \cite[Theorem 7.1]{GSM}.
In this case, the number of primitive TSRs of order one over
$\Fqm$ is given by
$$
  \frac{\left|\GL_m(\Fq)\right|}{(q^m-1)}\frac {\phi(q^m-1)}{m}.
$$
The case $m=1$ is trivial and in this case, the number of primitive
TSRs of order $n$ is given by 
$$\frac {\phi(q^{n}-1)}{n}.$$

However, for general values of $m$ and $n$, the enumeration of
primitive TSRs does not seem to be an easy problem and it still
stands open. Our focus in this paper is on irreducible TSRs.

\section{Irreducible TSRs} \label{TSRorder2}

For a given matrix $P$, let $\psi_P(X)$ denote the characteristic
polynomial of $P$. It follows from Lemma \ref{mntom} that for any
$T\in \TSR(m,n,q)$, the characteristic polynomial of $T$ is given by
\begin{equation} \label{form}
\psi_T(X)=g_T(X)^m\psi_B\left(\frac{X^n}{g_T(X)}\right),
\end{equation}
where $g_T(X)=1+c_1X+\cdots+c_{n-1}X^{n-1}\in \Fq[X]$. It is easy
to note that if $\psi_T(X)$ is irreducible, then so is $\psi_B(X)$,
but the converse is not true in general.

A TSR is \emph{primitive} (or \emph{irreducible}) if its characteristic
polynomial is primitive (or irreducible).
The set of irreducible TSRs is denoted by $\TSRI(m,n,q)$
and the set of monic irreducible polynomials in $\Fq[X]$ of degree
$d$ is denoted by $\mathcal I(d,q)$.
Then the \emph{characteristic map}
$$ \Psi: M_{mn}(\Fq) \longrightarrow \Fq[X]\quad \mbox{defined by}
        \quad\Psi(T):=\det(XI_{mn}-T)$$
if restricted to the set $\TSRI(m,n,q)$ 
yields the map
$$\Psi_I:\TSRI(m,n,q)\longrightarrow \mathcal I(mn,q).$$
It was noted in \cite{Ram} that the map $\Psi_I$ 
is not surjective in general.

The following lemma may be extracted from \cite{GSM} where it
is proved for primitive polynomials in some different context.
However, it still holds true even for irreducible polynomials.
We provide the proof of this lemma for irreducible polynomials
following similar lines as in \cite{GSM}. It turns out that
this may be viewed as an alternative proof of a special case
of \cite[Theorem 2]{Reiner}.
\begin{lemma} \label{NoMatrices}
Let $ \eta: M_{m}(\Fq) \longrightarrow \Fq[X]$ be defined by
$\eta(A):=\det(XI_{m}-A)$. Then, for every $p(X) \in \mathcal
I(m,q)$, we have,
$$\left|\eta^{-1}\left(p(X)\right)\right| =
  \displaystyle \prod_{i=1}^{m-1}(q^m-q^i).$$
\end{lemma}
\begin{proof}
Let us suppose that $T\in M_{m}(\Fq)$ be such that $\eta(T) = p(X)$.
Since $p(X)$ is irreducible, it is also the minimal polynomial
of $T$. The invariant factors of the companion matrix $C$ of
$p(X)$ and $T$ are the same and as a consequence they are similar
(see [\ref {BOUR}, p. VII.32]). It follows that $\eta^{-1}(p(X)) =
\{A^{-1} C A : A \in \GL_m(\Fq)\}$. Thus,
$$ \left|\eta^{-1}(p(X))\right| = \displaystyle
   \frac{\left|\GL_m(\Fq)\right|}{\left|Z(C)\right|},
   \quad \mbox{where}\quad
   Z(C):= \left\{A \in \GL_m(\Fq) : C A = A C \right\}.
$$

Now, $C$ as a linear transformation of $\Fqm\simeq \Fq^m$ is cyclic.
It follows from \cite[Theorem 3.16 and its corollary]{jac} that $Z(C)$
consists only of polynomials in $C$ excluding, however, the zero
polynomial. Thus $Z(C)=\Fq[C] \setminus \{0\}$, where $\Fq[C]$ is
the $\Fq$-algebra of polynomials in $C$.

The map $r(X)\mapsto r(C)$ defines a $\Fq$-algebra homomorphism from
$\Fq[X]$ into 
$\Fq[C]$ with kernel the ideal of $\Fq[X]$ generated by $p(X)$. Hence,
$\Fq[C]$ is isomorphic to $\Fq[X]/\left\langle p(X)\right\rangle$
and so its cardinality is $q^m$. Therefore, $|Z(C) |= q^m-1$, and
this completes the proof since $\left|\GL_m(\Fq)\right| =
\displaystyle \prod_{i=0}^{m-1}(q^m-q^i)$.
\end{proof}
It follows from \eqref{form} and \cite[Theorem $3$]{Ram} that
$f(X) \in \Psi_I\left(\TSRI(m,n,q)\right)$ if and only if $f(X)$
is irreducible and can be uniquely expressed in the form
\begin{equation} \label{irrform}
g(X)^m h\left(\frac{X^n}{g(X)}\right)
\end{equation}
for some monic irreducible polynomial $h(X)\in \Fq[X]$ of degree $m$
with $h(0)\neq 0$ and a not necessarily monic $g(X) \in \Fq[X]$ of
degree at most $n-1$ with $g(0)=1$.

\begin{theorem} \label{NoTSR}
The number of irreducible TSRs of order $n$ over $\Fqm$ is given
by the following
$$\left|\TSRI(m,n,q)\right|
  = \left|\Psi_I\left(\TSRI(m,n,q)\right)\right|
    \displaystyle \prod_{i=1}^{m-1}(q^m-q^i).$$
\end{theorem}
\begin{proof}
Let us assume that $f(X) \in \Psi_I\left(\TSRI(m,n,q)\right)$;
then $f(X)$ can be uniquely expressed in the form \eqref{irrform}.
Moreover, there is $T \in \TSRI(m,n,q)$ such that $\psi_T(X)=f(X)$.
Clearly $g_T(X)=g(X)$ and $\psi_B(X)=h(X)$. The number of such $T$
is equal to the number of possible values of $B$ with $\psi_B(X)=h(X)$.
Since $h(X)$ is irreducible, by Lemma \ref{NoMatrices}, the number
of such $B$ is $\displaystyle \prod_{i=1}^{m-1}(q^m-q^i)$.
\end{proof}
The case $m=1$ is trivial and in this case, the number of irreducible
TSRs of order $n$ is given by,
\begin{equation}
\displaystyle  \frac{1}{n}
     \sum_{d\mid n} \mu \left(d\right)q^{\frac{n}{d}}.
     \end{equation}
In the case $n=1$, the number of irreducible TSRs of order one is
given by
\begin{equation} \label{irrTSR1}
 \displaystyle \frac{1}{m} \prod_{i=1}^{m-1}(q^m-q^i)
     \sum_{d\mid m} \mu\left(d\right)q^{\frac{m}{d}}.
\end{equation}
In view of Theorem \ref{NoTSR}, it is sufficient to enumerate the
polynomials in the set $\Psi_I\left(\TSRI(m,n,q)\right)$ to find
the number of irreducible TSRs. In fact, Ram \cite{Ram} enumerates
TSRs of order two. Moreover, he re-derives a theorem of Carlitz
\cite{Carlitz} about the number of self reciprocal irreducible
monic polynomials of a given degree over a finite field. In this
section, we give a short proof of \cite[Theorem 8]{Ram} using a
generalization due to Ahmadi \cite{Ahmadi} of a result of Carlitz.
\begin{proposition}\label{Car} \cite{Ahmadi}
Let $e(X)=a_1X^2+b_1X+c_1$ and $g(x)=a_2X^2+b_2X+c_2$ be two
relatively prime polynomials in $\Fq[X]$ with $\max\left(\deg(e),
\deg(g)\right)=2$. Also let $\mathcal I(e,g,m,q)$ be the set of
monic irreducible polynomials $h(X)$ of degree $m>1$ over $\Fq$
such that
$${g(X)}^m h \displaystyle \left(\frac{e(X)}{g(X)}\right)$$
is irreducible over $\Fq$. Then
\[
    \left| \mathcal I(e,g,m,q)\right|=
    \begin{cases}
        0 & \text{if $b_1=b_2=0$ and $q$ is even}; \\
        \displaystyle \frac{1}{2m}(q^m-1)
          & \text{if $q$ is odd and $m=2^\ell$, $\ell \geq 1$}; \vspace{0.1cm} \\
        \displaystyle \frac{1}{2m}
        \displaystyle \sum_{d \mid m, \text{$d$ odd}} \mu(d)q^{\frac{m}{d}}
          & \text{otherwise}.
    \end{cases}
\]
\end{proposition}
We use the above proposition to give a short proof of \cite[Theorem 8]{Ram}
to count the number of irreducible TSRs of order two over $\Fqm$.
\begin{theorem}
\label{nequals2} For $m>1$, we have,
\[
    \left| \Psi_I\left(\TSRI(m,2,q)\right) \right|=
    \begin{cases}
       \displaystyle \frac{q}{2m} (q^m-1)
       & \text{if $q$ is odd and $m=2^\ell$}; \vspace{0.1cm} \\
       \displaystyle \frac{q}{2m}
       \displaystyle \sum_{d \mid m, \text{$d$ odd}} \mu(d)q^{\frac{m}{d}}
       &\text{if $q$ is odd and $m=2^{\ell}k$,} \vspace{-0.35cm} \\
       &\text{and $k \ge 3$ is odd};\\
       \displaystyle \frac{q-1}{2m}
       \displaystyle \sum_{d \mid m, \text{$d$ odd}} \mu(d)q^{\frac{m}{d}}
       &\text{otherwise}.
    \end{cases}
\]
\end{theorem}
\begin{proof}
For every $a \in \Fq$, let $\mathcal I_m(a)$ denote the set of monic
irreducible polynomials $h(X)$ of degree $m>1$ over $\Fq$ such that
$${(aX+1)}^m h \displaystyle \left(\frac{X^2}{aX+1}\right)$$
is irreducible over $\Fq$.
A direct application of Proposition \ref{Car} for $e(X)=X^2$ and
$g(X)=aX+1$ yields
\[
   \left|\mathcal I_m(a)\right|=
   \begin{cases}
       0 & \text{if $a=0$ and $q$ is even}; \\
         \displaystyle \frac{1}{2m}(q^m-1)
         & \text{if $q$ is odd and $m=2^\ell$, $\ell \geq 1$}; \vspace{0.1cm} \\
         \displaystyle \frac{1}{2m}
         \displaystyle \sum_{d \mid m, \text{$d$ odd}} \mu(d)q^{\frac{m}{d}}
         &\text{otherwise}.
    \end{cases}
\]
In view of (\ref{irrform}), the proof is complete after the
following observation
\[
   \left| \Psi_I\left(\TSRI(m,2,q)\right)\right|
         =\displaystyle \sum_{a\in \Fq}\left|\mathcal I_m(a) \right|=
 \begin{cases}
        \displaystyle \ |\mathcal I_m(1)|q & \text{if $q$ is odd}; \\
        \displaystyle \ |\mathcal I_m(1)|(q-1) & \text{if $q$ is even}.
 \end{cases}
\]
\end{proof}

Combining Theorem~\ref{NoTSR} and Theorem~\ref{nequals2}, we give
an alternative proof of Theorem 8 in \cite{Ram}.
\begin{theorem}
\label{irrTSR2}
For $m>1$, the number of irreducible TSRs of order two over
$\Fqm$ is given by
\[
 \left| \TSRI(m,2,q)\right|=
 \begin{cases}
        \displaystyle \frac{q}{2m}
        \displaystyle \prod_{i=0}^{m-1}(q^m-q^i)
        & \text{if $q$ is odd and $m=2^\ell$}; \\
        \displaystyle \frac{q}{2m}
        \displaystyle \prod_{i=1}^{m-1}(q^m-q^i)
       \displaystyle \sum_{d \mid m, \text{$d$ odd}} \mu(d)q^{\frac{m}{d}}
       &\text{if $q$ is odd, $m=2^{\ell}k$,} \vspace{-0.35cm}\\
           &   \text{and $k\ge 3$ is odd};\\
        \displaystyle \frac{q-1}{2m} 
        \displaystyle \prod_{i=1}^{m-1}(q^m-q^i)
        \displaystyle \sum_{d \mid m, \text{$d$ odd}} \mu(d)q^{\frac{m}{d}}
        &\text{otherwise}.
           \end{cases}
\]
\end{theorem}
\section{Asymptotic analysis of the number of irreducible TSRs of order two}
\label {Asymptotic}
Although we already know the explicit formula for the number of irreducible
TSRs of order two.
However, in this section we will be doing the asymptotic analysis for the
number of irreducible TSRs of order two by using some results due to
Cohen \cite{Cohen2}. For the convenience of the reader, we recall
here some notation and a theorem of Cohen about the distribution
of polynomials over finite fields \cite{Cohen2}.

Let $e, g \in \Fq[X]$ be monic relatively prime polynomials satisfying the following conditions:
\begin{enumerate}
  \item $n=\deg e > \deg g \geq 0$;
  \item $\displaystyle \frac{e(X)}{g(X)} \neq
        \frac{e_1(X^p)}{g_1(X^p)}$ for any $e_1, g_1 \in \Fq[X]$.
\end{enumerate}
Further, let $\G^{e,g}$ be the Galois group of $e(X)-tg(X)$
over $\displaystyle\Fqm(t)$, where $t$ is an indeterminate,
with splitting field $K$. We regard $\G^{e,g}$ as a subgroup
of $\mathcal S_n$, the $n$th symmetric group. Let
$\G^{e,g}_{\lambda}$ be the set of elements of $\G^{e,g}$
having the same cycle pattern $\lambda$. For any
$\sigma \in \G^{e,g}$, let $K_{\sigma}$ denote the subfield
of $K$ fixed under $\sigma$.

Moreover, let $\Fqm^{'} (=\displaystyle \mathbb F_{{(q^m)}^s}$
for some $s \geq 1$) be the largest algebraic extension of
$\Fqm$ in $K$. Let $\widehat{\G}{}^{e,g}=
\{\sigma \in \G^{e,g}~:~K_{\sigma}\cap \Fqm^{'}=\Fqm\}$ and put
${\widehat {\G}}{}^{e,g}_{\lambda}=
\widehat \G{}^{e,g} \cap \G^{e,g}_{\lambda}$ for any cycle
pattern $\lambda$. We note that $\sigma \in \widehat{\G}{}^{e,g}$
if and only if $K_{\sigma} \cap \Fqm^{'}(t)=\Fqm(t)$.

With these notations, we recall a lemma that is used
in the sequel \cite[Lemma 1]{Cohen2}.
\begin{lemma} \label{card} \cite{Cohen2}
With the notation as above, we have
$$ \displaystyle \left|\widehat{\G}{}^{e,g}\right| = \frac{\phi(s)}{s}\left|\G^{e,g}\right|,$$
where $\phi$ is Euler's totient function.
\end{lemma}

It is also mentioned in \cite{Cohen2} that if $\widehat{\G}{}^{e,g}$
is isomorphic to the symmetric group $\mathcal S_n$ and $\lambda$
is a cycle of order $n$, then
\begin{equation} \label{Sn}
\frac{\left|\widehat{\G}{}^{e,g}_{\lambda}\right|}{\left|\widehat{\G}{}^{e,g}\right|}=\frac{1}{n}.
\end{equation}

 Throughout this section, all the constants implied by $O$-terms
depend only on $n=\displaystyle \deg \left(e(X)-tg(X)\right)$.
\begin{proposition} \label{asymp}\cite{Cohen2}
Let $e, g \in \Fq[X]$ be as stated above.
Also let $\mathcal I(e,g,m,q)$ be the set of
monic irreducible polynomials $h(X)$ of degree $m$ over $\Fq$
such that
$${g(X)}^m h \displaystyle \left(\frac{e(X)}{g(X)}\right)$$
is irreducible over $\Fq$. Then
$$\left|\mathcal I(e,g,m,q)\right|=\displaystyle
  \frac{|\widehat{\G}{}^{e,g}_n|}{|\widehat{\G}{}^{e,g}|}
  \frac{q^m}{m}+O\left(q^{\frac{m}{2}}\right).$$
Moreover, when $\widehat{\G}{}^{e,g}=\mathcal S_n$,
$$\left|\mathcal I(e,g,m,q)\right|=\displaystyle
     \frac{1}{mn}q^m+O\left(q^{\frac{m}{2}}\right).$$
\end{proposition}
For $e(X)=X^n$ and $g(X)=1+a_1X+\cdots+a_{n-1}X^{n-1}$, we shall
alternatively denote the Galois group $\G^{e,g}$ of $X^n-tg(x) \in
\Fqm(t)[X]$ by $\G^{\bar a}$, where ${\bar a}=(a_1,\ldots,a_{n-1})~\in
\Fq^{n-1}$. Using this notation, we give a formula for the cardinality
of the set $\Psi_I\left(\TSRI(m,n,q)\right)$ and we further prove that
this is indeed an asymptotic formula in some special cases.
\begin{theorem} \label{asymptsr}
Let $m>1$ and $g(X)=1+a_1X+\cdots+a_{n-1}X^{n-1}$. 
Assume $\G^{\bar a}$
is the Galois group of $X^n-tg(X)$ over $\displaystyle \Fqm(t)$, where ${\bar a}=(a_1,\ldots,a_{n-1} )$.
Then for $n>1$, we have
$$ \left| \Psi_I\left(\TSRI(m,n,q)\right) \right|=
  c\frac{q^{m}}{m}
+ O \left(q^{n-1+\frac{m}{2}}\right),$$
where
$c= \displaystyle \sum_{{\bar a}\in \Fq^{n-1}}
    \frac{|\widehat {\G}{}^{\bar a}_n|}{|\widehat{\G}{}^{\bar a}|}$
and for $n=1$, we have
$$ \left| \Psi_I\left(\TSRI(m,n,q)\right) \right|=
   \frac{1}{m}q^{m}
 + O \left(q^{\frac{m}{2}}\right).$$
\end{theorem}
\begin{proof}
Assume that $n>1$ and for every ${\bar a}=(a_1,\ldots,a_{n-1} ) \in \Fq^{n-1}$,
let $\mathcal I_m({\bar a})$ denote the set of monic irreducible
polynomials $h(X)$ of degree $m>1$ over $\Fq$ such that
$${g(X)}^m h \displaystyle \left(\frac{X^n}{g(X)}\right)$$
is irreducible over $\Fq$, where $g(X)=1+a_1X+\cdots+a_{n-1}X^{n-1}$.
A direct application of Proposition \ref{asymp} with $e(X)=X^n$
and $g(X)=1+a_1X+\cdots+a_{n-1}X^{n-1}$ yields
$$ \left|\mathcal I_m({\bar a}) \right|= \displaystyle
   \frac{|\widehat{\G}{}^{\bar a}_n|}{|\widehat{\G}{}^{\bar a}|}
   \frac{q^m}{m}+O\left(q^{\frac{m}{2}}\right).$$
However, in the particular case when
$\widehat{\G}{}^{\bar a}=\mathcal S_n$, we have
$$ \left| \mathcal I_m(\bar a) \right|=
   \displaystyle \frac{1}{mn}q^m+O\left(q^{\frac{m}{2}}\right).$$
In view of (\ref{irrform}), we have
$$  \left| \Psi_I\left(\TSRI(m,n,q)\right) \right|
  = \displaystyle \sum_{{\bar a}\in \Fq^{n-1}}
    \left|\mathcal I_m({\bar a})\right|
  = c\frac{q^{m} }{m}
  + O \left(q^{n-1+\frac{m}{2}}\right),$$
where $c= \displaystyle \sum_{{\bar a}\in \Fq^{n-1}}
\frac{|\widehat{\G}{}_n^{\bar a}|}{|\widehat{\G}{}^{\bar a}|}$.

For $n=1$, we have $e(X)=X$ and $g(X)=1$. Thus,
$\widehat{\G}{}_1^{e,g}=\widehat{\G}{}^{e,g}=\G^{e,g}=\mathcal S_1$
and in this case, the proof follows from Proposition \ref{asymp}.
\end{proof}
We remark that in the proof of the above theorem, $g(X)$ is
not necessarily a monic polynomial, but we could still apply
Proposition \ref {asymp}.

The following theorem is an easy consequence of Theorem \ref{asymptsr}
and gives a formula for the number of irreducible TSRs.

\begin{theorem} \label{irrTSRAsymp}
Let us suppose that $m>1$. Then the number $\left| \TSRI(m,n,q)\right|$
of irreducible TSRs of order $n>1$ over $\Fqm$ satisfies
$$\left| \TSRI(m,n,q)\right|=c\frac{q^{m}}{m}
  \displaystyle \prod_{i=1}^{m-1}(q^m-q^i)
 + O\left(q^{m^2+n-1-\frac{m}{2}}\right),$$
where
$c= \displaystyle \sum_{{\bar a}\in \Fq^{n-1}}
    \frac{|\widehat {\G}{}^{\bar a}_n|}{|\widehat{\G}{}^{\bar a}|}$.
For $n=1$, we have \vspace{-0.3cm}
$$\left| \TSRI(m,n,q)\right|=\frac{1}{m} q^{m}
  \displaystyle \prod_{i=1}^{m-1}(q^m-q^i)
 + O \left(q^{m^2-\frac{m}{2}}\right).$$
\end{theorem}
\begin{proof}
The proof follows immediately from Theorem \ref{NoTSR} and
Theorem \ref{asymptsr}.
\end{proof}
\begin{remark}
The explicit computation of the constant $c$ in Theorem \ref{asymptsr}
seems a rather difficult problem. Without knowing the behaviour of $c$,
it is not clear if the $c\frac{q^{m}}{m}$ term can be absorbed into the
big Oh term; if this happens, we no longer have an asymptotic formula.
When $m < 2(n-1)$, it is not clear if $cq^m$ is asymptotically bigger
than $q^{n-1+m/2}$. Thus, unless we know the asymptotics of $c$ as a
power of $q$ for large values of $q$, Theorem \ref{asymptsr} does not
give an asymptotic formula for $\left| \Psi_I\left(\TSRI(m,n,q)\right)
\right|$. The same holds true for Theorem \ref{irrTSRAsymp}.

\end{remark}


It is clear that for $n=1$, the first term ($d=1$) in (\ref{irrTSR1})
is exactly the same as the main term in the formula of Theorem
\ref{irrTSRAsymp}. For the case $n=2$, we explicitly compute the
value of $c$ in the following theorem allowing us to compare the
main term in the formula of Theorem \ref{irrTSRAsymp} with the
first term in the formula of Theorem \ref{irrTSR2}. When $n=2$,
we prove that the main terms $c\frac{q^{m}}{m}$ and
$c\frac{q^{m}}{m} \displaystyle \prod_{i=1}^{m-1}(q^m-q^i)$ of
Theorem \ref{asymptsr} and Theorem \ref{irrTSRAsymp}, respectively,
do not get absorbed in the big Oh term.

\begin{theorem} \label{valueofc}
Let $p$ be the characteristic of the field $\Fqm$. For $n=2$,
the value of the constant $c$ in Theorem $\ref{irrTSRAsymp}$
is $\frac{q}{2}$ whenever $p \neq 2$, and $\frac{q-1}{2}$ if $p=2$.
\end{theorem}
\begin{proof}
For $n=2$, we have $e(X)=X^2$, $g(X)=aX+1$, and ${\bar a}=a \in \Fq$.
We consider two different cases depending upon the characteristic
$p$ of the field $\Fqm$.

{\bf Case $1$:}  Suppose $ p \neq 2$. Then for each ${\bar a}=a$
in $\Fq$, $X^2-t(aX+1)$ is irreducible and separable over
$\Fqm(t)$ and thus $\G^{\bar a}=\mathcal S_2$.

Let $K$ be splitting field of $X^2-t(aX+1)$ over $\Fqm(t)$
and let $\Fqm^{'} (=\mathbb F_{{(q^m)}^s}$ for some $s \geq 1$)
be the largest algebraic extension of $\Fqm$ in $K$. We have
$\Fqm(t) \subseteq \Fqm^{'}(t) \subseteq K$. Since
$[K:\Fqm(t)]=2$, $\Fqm^{'}(t)$ is either equal to $K$ or
$\Fqm(t)$. But the irreducibility of the polynomial $X^2-t(aX+1)$
over $\Fqm^{'}(t)$ ensures that   $\Fqm^{'}(t) \neq K$.
Therefore $\Fqm^{'}(t)=\Fqm(t)$ and hence, $s=1$. Thus
using Lemma \ref {card}, we have
$\widehat {\G}{}^{\bar a}=\G^{\bar a}=\mathcal S_2$.
Now by using \eqref{Sn}, we obtain
$$c= \displaystyle \sum_{{ \bar a=a}\in \Fq}
     \frac{|\widehat {\G}{}^{\bar a}_2|}
         {|\widehat{\G}{}^{\bar a}|}=\frac{q}{2}.$$

{\bf Case $2$:} Suppose $p=2$. Then for each ${\bar a}=a \neq 0$
in $\Fq$, $X^2-t(aX+1)$ is irreducible and separable over
$\Fqm(t)$ and thus $\G^{\bar a}=\mathcal S_2$. Following
similar arguments as before, we deduce that  for ${\bar a} = a \neq 0$,
$\widehat {\G}{}^{\bar a}=\G^{\bar a}=\mathcal S_2$.

However, when ${\bar a} = a =0$, the polynomial $x^2-t$ is
irreducible, but  not separable over $\Fqm(t)$. Thus,
$\widehat {\G}{}^{0}=\G^{0}=\mathcal A_2$ and hence
$\left|\widehat{\G}{}^{0}_2\right|=0$. Again Equation
\eqref{Sn} yields
$$c= \displaystyle \sum_{{ \bar a=a}\in \Fq}
     \frac{|\widehat {\G}{}^{\bar a}_2|}{|\widehat{\G}{}^{\bar a}|}
   = \displaystyle \sum_{{ \bar a=a \neq 0}\in \Fq}
     \frac{|\widehat {\G}{}^{\bar a}_2|}{|\widehat{\G}{}^{\bar a}|}
   = \frac{q-1}{2}.$$ 
\end{proof}

\section{An asymptotic formula for the number of irreducible TSRs of any order when $q$ is odd}
\label{Asymptoticodd}
In this section, we prove an asymptotic formula for the number of irreducible
TSRs of any order when $q$ is odd by using some previous results due to Cohen \cite{Co99}.

It may be noted that $f$ is necessarily monic  of degree $mn$ in
\eqref{irrform} and $f(0)=h(0) \neq 0$. Its (monic) reciprocal is
$f^*(X)= X^{\deg f}f(1/X)/f(0)$. Of course, $f$ is irreducible if
and only if $f^*$ is irreducible. From \eqref{irrform}
\begin{equation} \label{fgh}
 f^*(X)=X^{mn}g(1/X)^m h\left(\frac{1}{X^ng(1/X)}\right)/f(0)=h^*(\bar{g}^*(X)),
\end{equation}
since $f(0)=h(0)\neq 0$ and $\bar{g}(X) = X g(X)$. Thus, from now on,
if we replace $g$ by the reciprocal $X^n+a_1X^{n-1}+ \cdots + a_{n-1}X$
of $\bar{g}$, we have that $M(m,n,q):=|\Psi_I(\mathrm{TSRI}(m, n, q))|$
is the number of irreducible polynomials in $\mathbb{F}_q[X]$ of the form
$h(g(X))$, where $h$ is a monic polynomial of degree $m$ (necessarily
irreducible) and $g$ is a monic polynomial of degree $n$ (with $g(0)=0$),
as described.  Suppose $ \alpha$ is a root in $\mathbb{F}_{q^m}$ of a
monic irreducible polynomial $h(X) \in \mathbb{F}_q[X]$ of degree $m$.
Then $h(g(X))$ is irreducible in $\mathbb{F}_q[X]$ if and only if
$g(X)-\alpha$ is irreducible in $\mathbb{F}_{q^m}[X]$. Hence $mM(m,n,q)$
is sum over all $(n-1$)-tuples $\bar{a}$ of the number of
$\alpha \in \mathbb{F}_{q^m}$, not in a proper subfield, such that
$g(X)-\alpha$ is irreducible in  $\mathbb{F}_{q^m}$.

When $m=1$, then  $M(1,n,q)$ is simply the number of irreducible
polynomials of degree $n$ over $\mathbb{F}_q$, given by the well-known
formula. So suppose $m>1$ and define  $N(m,n,q)$ to be the sum over
$\bar{a}$ of  the total number of $\alpha \in \mathbb{F}_{q^m}$ such
that $g(X)- \alpha$ is irreducible in $\mathbb{F}_{q^m}$. Then
 \begin{equation} \label{MN}
  N(m,n,q)=mM(m,n,q)+O(q^{n-1+m/2}).
  \end{equation}

Let $K_q$ be the algebraic closure of the field $\mathbb{F}_q$ (and so
of $\mathbb{F}_{q^m}$). Let $F(X)=g(X)-t$, where $t$ is an indeterminate.
For given $\bar{a}$, $\G^{\bar{a}}$ denotes the Galois group of $g(X)-t$
over $\mathbb{F}_{q^m}(t)$, where $t$ is an indeterminate. It has as a
normal subgroup $\widehat{\G}{}^{\bar{a}}$, the Galois group of $g(X)-t$
over $K_q(t)$.  An important criterion for $\widehat{\G}{}^{\bar{a}}$ to be
the full symmetric group $\mathcal S_n$ derives from Theorem 4.8 of \cite{Co99}.

\begin{lemma} \label{GalSn}
Let $g(X) \in \mathbb{F}_q[X]$ be monic of degree $n$ and indecomposable
over $\mathbb{F}_q$ (i.e, $g$ is not a composition
$g=g_1(g_2)$ of polynomials $g_1(X), g_2(X) \in \mathbb{F}_q[X]$, where $\deg(g_i) \geq 2,\ i=1,2$). Suppose that, for
some $\theta \in K_q$, $g(X)-\theta$ factorizes over $K_q$ as
$(X-\beta)^2E(X)$ for some square-free polynomial $E$ (with
$E(\beta) \neq 0$). Then the Galois group of $g(X)-t$ over $K_q(t)$
is $\mathcal S_n$.
\end{lemma}

We can suppose $n \geq 3$.  It turns out we have to exclude from
consideration $(n-1)$-tuples  $\bar{a}$ of a certain form as we
now describe. Let $p$ be the characteristic of $\mathbb{F}_q$, i.e.,
$q$ is a power of the prime $p$. The polynomial
$g(X) = X^n+a_1X^{n-1}+ \cdots + a_{n-1}X$ is said to be of form
(\ref{AB}) if we can  express it in the form
\begin{equation}\label{AB}
   XA(X^p)+B(X^p),
\end{equation}
where $A,B$ are polynomials, i.e., $n \equiv 0, 1 ~(\mathrm{mod}\ p)$
and $a_i=0$, whenever $i \not \equiv 0, 1 ~(\mathrm{mod}\ p)$. Given
$a_1, \ldots, a_{n-2} \in \mathbb{F}_q$, set
$F_0(X)=g(X)-a_{n-1}X=X^n+\sum_{i=1}^{n-2}a_iX^{n-i}$. Observe that
$F_0$ has form (\ref{AB}) if and only if $g$ has form (\ref{AB})
for any $a_{n-1} \in \mathbb{F}_q$.

We remark further that if $p=2$, then every polynomial $g$ has the
form (\ref{AB}). Hence, it is necessary from now to impose the
restriction that $q$ is odd.

\begin{lemma} \label{simple}
Suppose $q$ is odd and $n \geq 3$. Let $a_1, \ldots,a_{n-2}$ be any
elements of $\mathbb{F}_q$ such that $F_0$ does not have the form
$(\ref{AB})$. Then, for all but $O(1)$ choices of non-zero elements
$a_{n-1} \in \mathbb{F}_q$, $\widehat{\G}{}^{\bar{a}}=\mathcal S_n$. (Here, as
throughout, the implied constant depends only on $n$.)
\end{lemma}
\begin{proof}
It has to be shown that, for all but $O(1)$ choices of $a_{n-1}$, $g$ is indecomposable
over $\mathbb{F}_{q^m}$ and,
for any $\theta \in K_q$, either $g(X)- \theta$ is square-free
or factorizes as $(X-\beta)^2E(X)$, as described in  Lemma $\ref{GalSn}$.
The proof of this follows exactly that of Lemma 5 of \cite{Co72},
in the special case in which $s=2$ and the polynomials $F_0, F_1, F_2$
(in the notation of Theorem 3 of \cite{Co72}) are, respectively, $F_0$
as defined here, $F_1(X)=X, F_2(X)=1$. The proof of \cite{Co72},
Lemma 5, is derived from that of Lemmas 6, 7, and the identical
arguments can be used in this particular situation. (Note, in particular,
that assumption $p\nmid n$ of \cite{Co72}, Theorem 3, is not required
at this stage.) The main thrust of the proof of \cite{Co72}, Lemma 6, is that
with $O(1)$ exceptional values of $a_{n-1}$, $g$ is indecomposable (actually even over
$K_q$). Otherwise, $F_0,F_1,F_2$ would be  ``totally composite", which is evidently not the case.
Further, the assumption that $F_0, F_1, F_2$ are linearly independent
over $\mathbb{F}_{q^m}(X^p)$ of \cite{Co72}, Theorem 3, in our situation,
is a consequence of the assumption that $F_0$ does not have form (\ref{AB}).

The conclusion of \cite{Co72}, Lemma 7, is that if $a_{n-1}$ is one
of the $q-O(1)$ (non-zero) elements of $\mathbb{F}_q$ that have not
been excluded, then, for every $\theta \in K_q$, either $g(X)- \theta$
is square-free or has the form $(X-\beta)^2E(X)$. Now, let
$\beta \in K_q$ be any root of the formal derivative $g'(X)$. Indeed,
since $g$ does not have the form (\ref{AB}), there is such an element
$\beta$. Set $\theta = g(\beta)$. Then $\beta$ is a repeated root of
$g(X) - \theta$ of multiplicity $2$ and there are no other repeated
roots of $g(X) - \theta$. Then Lemma \ref{GalSn} applies and we
conclude that $\widehat{\G}{}^{\bar{a}}= \mathcal S_n$.
\end{proof}

\begin{theorem} \label{Nthm}
Suppose $q$ is odd, $n \geq 3$ and $ m \geq 2$. Then
\[ N(m,n, q) =\frac{q^{m+n-1}}{n} +O(q^{m+n-2}).\]
\end{theorem}
\begin{proof}
There are in total $q^{n-1}$ choices of $\bar{a}$ in the polynomial $g$.
We show that for all but $O(q^{n-2})$ of them $\widehat{\G}{}^{\bar{a}} =\mathcal S_n$,
whence, by \cite[Theorem 1]{Cohen2}
for every non-excluded choice $\bar{a}$, the number of
$\alpha \in \mathbb{F}_{q^m}$ such that $g(X)-\alpha$ is irreducible is
\begin{equation}\label{coh}
\frac {q^m}{n} +O(q^{m/2}).
\end{equation}
Given $a_1, \ldots, a_{n-2}$ in $\mathbb{F}_q$, let the implied constant
in the number of values of $a_{n-1}$ to be excluded be bounded above by
$d (=d_n)$. Altogether, this excludes at most $dq^{n-2}$ choices of $\bar{a}$.
When $n \not \equiv 0, 1 ~(\mathrm{mod}\ p)$, by Lemma \ref{simple}, for
the remaining choices of $\bar{a}$, $\widehat{\G}{}^{\bar{a}} =\mathcal S_n$ and, by
(\ref{coh}),
\begin{equation} \label{excl}
N(m,n,q) \geq \frac {q^{m+n-1}-dq^{m+n-2}}{n} +O(q^{n-1+m/2})
  = \frac{q^{m+n-1}}{n} +O(q^{m+n-2}).
\end{equation}
When $n \equiv 0, 1 ~(\mathrm{mod}\ p)$, further values of $\bar{a}$
have to be excluded because, in Lemma \ref{simple}, $g$ has the form
$(\ref{AB})$. In particular, when $p|n$, then these further excluded
values all have $a_1=0$, whence their total number does not exceed
$q^{n-2}$. Similarly, if $ n \geq 3$ and $n \equiv 1 ~(\mathrm{mod} \ p)$,
then $n \geq p+1 \geq 4$ and all further excluded  $\bar{a}$ have $a_2=0$.
Thus their total number again does not exceed $q^{n-2}$. The argument in
these cases then proceeds as at (\ref{excl}) with $d$ replaced by $d+1$.
\end{proof}
\begin{corollary} \label{cor}
Suppose $q$ is odd, $n \geq 3$ and $m \geq 2$. Then
\[ M(m,n,q)=|\Psi_I(\mathrm{TSRI}(m, n, q))|
  =\frac{q^{m+n-1}}{mn} +O(q^{m+n-2}/m).\]
\end{corollary}
\begin{proof}
This follows from Theorem \ref{Nthm}, along with $(\ref{MN})$ and the
definition of $M$.
\end{proof}

From Corollary \ref{cor}, when $q$ is odd, for $q >q_n$ the constant $c$
in Theorem \ref{asymptsr} is positive.

\begin{theorem} \label{irrTSRAsympoddq}
Suppose that $q$ is odd and $m>1$. Then the number $\left| \TSRI(m,n,q)\right|$
of irreducible TSRs of order $n>2$ over $\Fqm$ satisfies
$$\left| \TSRI(m,n,q)\right|=\frac{q^{m+n-1}}{mn}
  \displaystyle \prod_{i=1}^{m-1}(q^m-q^i)
 + O\left(q^{m^2+n-2}/m\right).$$
\end{theorem}
\begin{proof}
The proof follows immediately from Theorem \ref{NoTSR} and Corollary \ref{cor}.
\end{proof}

We note that the main term in Theorem~\ref{Nthm} corresponds to the main
term in  Theorem~\ref{asymptsr}, however, the error term is slightly
increased in most of the cases. It may be interesting to determine if the
formula in Theorem \ref{asymptsr} and hence in Theorem \ref{irrTSRAsymp}
is asymptotic in nature when $q$ is even.
\section*{Acknowledgments}
Sartaj Ul Hasan and Qiang Wang would like to thank Daqing Wan for some
helpful discussions.

\end{document}